\documentclass{amsart}

\usepackage{amsmath}
\usepackage{amsfonts,amssymb}
\usepackage[latin1]{inputenc}
\usepackage{amsthm}
\usepackage{verbatim}
\usepackage{array}
\usepackage{graphicx,color}
\usepackage{graphics}

\newtheorem{thm}{Theorem}[section]

\newtheorem{lemma}[thm]{Lemma}

\theoremstyle{definition}

\theoremstyle{remark}
\newtheorem{rem}[thm]{Remark}

\newcommand{\R}{\mathbb{R}}

\newcommand{\dint}{\ds \int}

\newcommand{\ds}{\displaystyle}

\newcommand{\eqskip}{ \vspace*{2mm}\\ }
\newcommand{\hyp}{\mathbb{H}}

\newcommand{\om}{ \Omega }

\newcommand{\sph}{ \mathbb{S} }
\newcommand{\s}{{\rm s}}

\def\om{\omega}
\def\e{\varepsilon}

\def\l{\lambda}

\def\r{\rho}

\def\Odr{{\rm O}}
\def\odr{{\rm o}}

\def\Hloc{W_{2,loc}}

\def\di{\,\mathrm{d}}

\def\LL{\mathcal{L}}
\def\HH{\mathcal{H}}
\def\RR{\mathcal{R}}

\DeclareMathOperator{\Dom}{\mathcal{D}}

\numberwithin{equation}{section}

\title{The spectrum of geodesic balls on spherically symmetric manifolds}
\author{Denis Borisov and Pedro Freitas}
\address{Institute of Mathematics CC USC RAS, Chernyshevsky str. 112, 450008, Ufa, Russia \&
Bashkir State Pedagogical University, October St.~3a, 450000 Ufa,
Russia \&  University of Hradec Kr\'alov\'e, Rokitansk\'eho 62, 500 03, Hradec Kr\'alov\'e, Czech Republic}\email{borisovdi@yandex.ru}
\address{Department of Mathematics, Faculty of Human Kinetics \&
Group of Mathematical Physics, Faculty of Sciences, University of Lisbon, Campo Grande, Edif\'{\i}cio C6,
1749-016 Lisboa, Portugal}
\email{psfreitas@fc.ul.pt}
\date{\today}

\begin{document}

\allowdisplaybreaks

\begin{abstract}
We study the Dirichlet spectrum of the Laplace operator on geodesic
balls centred at a pole of spherically symmetric manifolds. We
first derive a Hadamard--type formula for the dependence of the first eigenvalue
$\lambda_{1}$ on the radius $r$ of the ball, which allows us to obtain
lower and upper bounds for $\lambda_{1}$ in specific cases. For the
sphere and hyperbolic space, these bounds are asymptotically sharp as
$r$ approaches zero and we see that while in two dimensions $\lambda_{1}$
is bounded from above by the first two terms in the asymptotics for
small $r$, for dimensions four and higher the reverse inequality holds.

In the general case we derive the asymptotic expansion of $\lambda_{1}$
for small radius and determine the first three terms explicitly.
For compact manifolds we carry out similar calculations as the radius of
the geodesic ball approaches the diameter of the manifold. In the
latter case we show that in even dimensions there will always exist
logarithmic terms in these expansions.
\end{abstract}

\maketitle

\section{Introduction}

Within the last forty years several papers appeared
in the literature devoted to the study of the first eigenvalue of
geodesic disks on spherically symmetric manifolds, with particular emphasis on the
case of spaces of constant curvature -- see, for example,~\cite{bagi,babe,bcg,came94,frha,gage,grig,katz,mcke,mata,pins78,pins94,sato,wang}.
In this particular instance, namely, hyperbolic space $\hyp^{n}$ and spheres
$\sph^{n}$, the solutions are known explicitly in terms of Legendre
functions of the first kind, the eigenvalues then being given in terms of zeros of such
functions. What is thus of interest is to obtain bounds and approximations which
may be written explicitly in terms of more elementary functions, as in general
computer packages may determine these zeros with the needed accuracy -- it should, however,
be noted that in limit cases such as when the radius becomes very large in hyperbolic space
these computations may still pose some difficulties from a numerical point of view.

For general spherically symmetric manifolds, however, it will not be possible to write the
eigenvalues of geodesic balls explicitly in terms of {\it known} functions, and it then becomes
important to have accurate estimates for these quantities. Apart from their intrinsic interest,
bounds of this type may also be used to estimate eigenvalues of balls on manifolds
which are not necessarily spherically symmetric, by making use of the recent spectral
comparison results established in~\cite{frmasa}.

A major difference when moving away from the Euclidean framework is that scaling the domain
no longer translates into a mere scaling in spectral terms. More precisely, while in the former
case we have $\lambda(\alpha \Omega) = \alpha^{-2}\lambda(\Omega)$ for a scaling of a domain $\Omega$
by a positive number $\alpha$ and any eigenvalue $\lambda$, for the latter the variation of eigenvalues
will provide extra information on the metric.

With the above in mind, it makes thus sense as a first step in this direction to derive a
Hadamard--type formula for the variation of the first eigenvalue of a geodesic ball with respect
to the radius. This will in turn allow us to derive some new lower and upper bounds which,
apart from improving existing results within certain ranges of the
radius, also have the advantage of being explicit in the sense mentioned
above. As an example of this, Theorem~\ref{thm:hypsph} gives lower and upper
bounds for disks in hyperbolic space and spheres which agree with the
first two terms in the asymptotics of the first eigenvalue as the radius
approaches zero.

In fact, and as a consequence of these bounds, we see that for (non-Euclidean) constant
curvature spaces, while in two dimensions the first eigenvalue is bounded from above by the first two
terms in the asymptotics, this relation is reversed for dimensions greater than
or equal to $4$ -- for $\hyp^{3}$ and $\sph^{3}$ the expressions obtained are
exact.

In the case of a general $n-$manifold, the first two terms in these asymptotic expansions
of the $k^{\rm th}$ eigenvalue $\lambda_{k}$ are known to be given by
\[
\lambda_{k}(r) = \frac{\ds 1}{\ds r^2}\gamma_{k} - \frac{\ds 1}{\ds 6} \mathcal{S}(p)+ \odr(1),
\]
where $\gamma_{k}$ denotes the $k^{\rm th}$ eigenvalue of the unit disk in $n-$dimensional Euclidean
space and $\mathcal{S}(p)$ denotes the scalar curvature at the point $p$ -- see~\cite[p. 318]{chav}.
In the particular case of a spherically symmetric manifold we shall derive the full asymptotic expansion for
the first eigenvalue of a disk of small radius $r$ centred at a pole and, for compact manifolds, also the expansion as this radius
approaches the diameter of the manifold. We then determine the expression for the first three
terms explicitly, from which we obtain, for instance, that under certain natural smoothness assumptions
the third term in the expansion above is of order $r^2$ and the corresponding expansion for $\lambda_{1}$ is given by
\[
\lambda_{1}(r) =  \frac{\ds j_{\frac{n}{2}-1,1}^2}{\ds r^2} - \frac{\ds 1}{\ds 6} \mathcal{S}(p)+ \left[\alpha_{1} \mathcal{S}^2(p) + \alpha_{2}
\mathcal{S}''(p)\right]r^2 + \odr(r^4),
\]
where $\alpha_{1}$ and $\alpha_{2}$ are constants which depend only on the dimension -- see Theorem~\ref{thmexpr0} below for the
details, including explicit expressions for these coefficients. The expressions for the asymptotics as the radius of the
ball approaches the diameter are more involved (in particular, they depend in a nontrivial way on the dimension) and are
presented in Theorem~\ref{th5.1}. Although this situation corresponds to the well--known singular perturbation problem of a manifold with a
small hole whose volume approaches zero -- see~\cite{fluc,chav} and~\cite{manapl}, for instance, and the references therein --
we shall see that in all even dimensions logarithmic terms do appear in the expansions for $\lambda_{1}$. This comes as a
surprise as, to the best of our knowledge, so far only in two dimensions was it known that a logarithmic term would be present (in
fact the leading term), as a direct consequence of the singularity of the corresponding Green's function -- see also Table II in~\cite{fluc}
which contains an overview of the results for this type of problems.

The plan of the paper is as follows. In the next section we fix the notation and state
some basic facts which will be used in the sequel. Section~\ref{sec:hadamard}
contains the statement and derivation of the Hadamard-type formula, followed by its application to obtaining upper and
lower bounds in Section~\ref{sec:appl}. We then proceed to determine the asymptotic expansions
mentioned above for small and maximal radius in Sections~\ref{asymptr0} and~\ref{asymptrmax}, respectively.

\section{Notation and preliminaries\label{notprelim}}

Given a sufficiently smooth function $f:\R^{+}_{0}\to\R^{+}_{0}$ satisfying
\begin{equation}\label{2.0}
f(0)=0,\quad f'(0)=1,\quad f''(0)=0,
\end{equation}
and $f(r)>0$ for all $r$ in $(0,R)$ for some positive $R$ (possibly infinite),
let $M$ be the spherically symmetric $n-$manifold with metric $dr^2+f^2(r)d\theta^{2}$ and $p$ be the point at
the {\it North pole}. The above restriction on the second derivative of $f$ stems from the fact that the scalar curvature
$\mathcal{S}$ at a point $p$ is given by~\cite[p. 69]{pete06}
\begin{equation}\label{exprscalarcurv}
 \mathcal{S}(p) = -2(n-1) \frac{\ds f''(t)}{f(t)} + (n-1)(n-2) \frac{\ds 1-\left[f'(t)\right]^2}{\ds f^2(t)}
\end{equation}
and it thus follows that for this to be finite at $t=0$ one must have $f''(0)=0$. It turns out that for the metric
to be regular at $t$ equal to zero stronger restrictions have to be imposed on $f$. In particular, if it is to be smooth,
then all even derivatives of $f$ must vanish at zero~\cite[pp. 12--13]{pete06}. Throughout the paper we shall make this assumption,
although it is clear that if $f^{(2k)}(0)$ is not zero for some $k$ larger than one then the terms appearing in
the asymptotics in Section~\ref{asymptr0} should be modified accordingly.

We consider the geodesic ball $B(r)$ centred at $p$ and with radius $r$. The $n-$volume of this ball is then given by
\begin{equation*}
V(r) = \omega_{n-1}\dint\limits_{0}^{r}f^{n-1}(t) dt,
\end{equation*}
while the $(n-1)-$volume of its boundary is given by
\[
S(r) = \omega_{n-1}f^{n-1}(r),
\]
where $\omega_{n-1}$ denotes the $(n-1)-$volume of the unit sphere $\sph^{n-1}$.

In this setting, the first eigenvalue of the Laplace-Beltrami operator on $B(r)$ with Dirichlet boundary conditions is simple and the corresponding
eigenfunction does not change sign and is also spherically symmetric. Denoting this eigenvalue by $\lambda=\lambda(r)$ and by $\psi$ a corresponding
eigenfunction, we have that the pair $(\lambda,\psi)$ satisfies
\begin{equation}\label{eqn:basic}
\left\{
\begin{array}{ll}
-\left[f^{n-1}(t)\psi'(t)\right]' = \lambda f^{n-1}(t)\psi(t), & t\in(0,r)\eqskip
\psi'(0)=\psi(r)=0. &
\end{array}
\right.
\end{equation}

We can also interpret $\l$ and $\psi$ as the first eigenvalue and the associated eigenfunction of the operator $\HH_r$, where $\HH_r$ is introduced as
the self-adjoint operator in $X_r^0$ associated with the closed lower-semibounded quadratic form $h_r[u]:=\|u'\|_{X_r^0}^2$ on $X_r$. Here
\begin{align*}
&X_r:=\{u\in\Hloc^1(0,r):\  u(r)=0,\ \|u\|_{X_r}<+\infty\},\quad \|u\|_{X_r}^2:=\|u'\|_{X_r^0}^2+\|u\|_{X_r^0}^2,
\\
&\|u\|_{X_r^0}^2:=\int\limits_{0}^{r} f^{n-1}(t)|u(t)|^2\di t,\quad X_r^0:=\{u\in L_{2,loc}(0,r): \|u\|_{X_r^0}<+\infty\}.
\end{align*}
In this way, $\lambda$ may also be obtained via its variational formulation which now becomes
\begin{equation}\label{eqn:variational}
\lambda(r) = \inf_{u\in X_{r}}\frac{\ds \dint_{0}^{r} f^{n-1}(t)\left[u'(t)\right]^2 dt}
{\ds \dint_{0}^{r} f^{n-1}(t)u^{2}(t)dt}.
\end{equation}

\section{A Hadamard-type formula for the first eigenvalue\label{sec:hadamard}}
While in the Euclidean case the first eigenvalue of balls centred at any point and with different radii are related to each other by
a simple rescaling, this will not be the case in more general ambient spaces. As a consequence, the way in which the eigenvalue varies
as a function of the radius of the ball will contain information about the ambient space. The purpose of this section is to establish
a formula for such a variation in the more general situation of a spherically symmetric manifold with the centre of the ball placed at
the North pole.
\begin{lemma}\label{lem:basic1}
Let $\psi$ be a solution of equation~(\ref{eqn:basic}) associated to the first eigenvalue
$\lambda(r)$. Then, for any non--negative $r_{0}$,
\[
\lambda(r) = \frac{\ds 1}{\ds r^2}\lim_{s\to r_{0}}\left[s^2\lambda(s)\right]
+ \frac{\ds n-1}{\ds 2r^2}\dint_{r_{0}}^{r} \left[\frac{\ds t \dint_{0}^{t}H(s) f^{n-1}(s)\psi^{2}(s) ds}{\ds \dint_{0}^{t} f^{n-1}(s)\psi^{2}(s) ds}
\right]dt,
\]
where
\[
\begin{array}{l}
H(t) = \left[t h'(t)+2h(t)\right] \eqskip
h(t) = g'(t)+\frac{\ds n-1}{\ds 2} g^{2}(t) \eqskip
g(t) = \frac{\ds f'(t)}{\ds f(t)}.
\end{array}
\]
In the particular case where $r_{0}$ is zero, the above expression simplifies to
\[
\lambda(r) = \frac{\ds j_{\frac{n}{2}-1,1}^2}{\ds r^2}
+ \frac{\ds n-1}{\ds 2r^2}\dint_{0}^{r} \left[\frac{\ds t \dint_{0}^{t}H(s) f^{n-1}(s)\psi^{2}(s) ds}{\ds \dint_{0}^{t} f^{n-1}(s)\psi^{2}(s) ds}
\right]dt,
\]
where $j_{\frac{n}{2}-1,1}$ is the first zero of the Bessel function $J_{\frac{n}{2}-1}$.
\end{lemma}
\begin{rem}
Note that the eigenfunction $\psi$ also depends on $r$ and that $H$ depends (linearly) on $n$.
\end{rem}
\begin{proof} From the variational formulation~(\ref{eqn:variational}) we have
\[
\lambda(\alpha r) = \inf_{u\in X_{\alpha r}}\frac{\ds \dint_{0}^{\alpha r} f^{n-1}(t)\left[u'(t)\right]^2 dt}
{\ds \dint_{0}^{\alpha r} f^{n-1}(t)u^{2}(t)dt}=\frac{\ds 1}{\ds \alpha^2}
\inf_{u\in X_{r}}\frac{\ds \dint_{0}^{r} f^{n-1}(\alpha t)\left[u'(t)\right]^2 dt}
{\ds \dint_{0}^{r} f^{n-1}(\alpha t)u^{2}(t)dt}
\]
and
\begin{equation}\label{eqn:raylquot}
\lambda(\alpha r)\leqslant
\frac{\ds 1}{\ds \alpha^2}\frac{\ds \dint_{0}^{r} f^{n-1}(\alpha t)\left[\psi'(t)\right]^2 dt}
{\ds \dint_{0}^{r} f^{n-1}(\alpha t)\psi^{2}(t)dt}
\end{equation}
where $\psi$ is a first eigenfunction corresponding to $\lambda$, that is, $\psi$ and $\lambda=\lambda(r)$ satisfy equation~(\ref{eqn:basic}).

Multiply now equation~(\ref{eqn:basic}) by $f^{n-1}(\alpha t)\psi(t)/f^{n-1}(t)$ and integrate
between $0$ and $r$ by parts twice to obtain
\[
\begin{array}{lll}
\dint_{0}^{r}f^{n-1}(\alpha t)\psi^2(t)dt\; \lambda(r) & = & -\dint_{0}^{r}
\frac{\ds f^{n-1}(\alpha t)\psi(t)}{\ds f^{n-1}(t)}\left[f^{n-1}(t)\psi'(t)\right]'dt\eqskip
& = & (n-1)\dint_{0}^{r}\left[\alpha f^{n-2}(\alpha t)f'(\alpha t)\right.\eqskip
& & \left.\hspace*{0.5cm}-f^{n-1}(\alpha t) \frac{\ds f'(t)}{\ds f(t)}\right]\psi(t)\psi'(t)dt\eqskip
& & \hspace*{1cm}+\dint_{0}^{r}f^{n-1}(\alpha t)\left[\psi'(t)\right]^2 dt\eqskip
& = & -\frac{\ds n-1}{\ds 2}\dint_{0}^{r}\left[\alpha^2(n-2)\left[\frac{\ds f'(\alpha t)}
{\ds f(\alpha t)}\right]^2 + \alpha^2\frac{\ds f''(\alpha t)}
{\ds f(\alpha t)}\right.\eqskip
& & \hspace*{0.5cm}\left.-\alpha(n-1)\frac{\ds f'(t)f'(\alpha t)}
{\ds f(t)f(\alpha t)}-\frac{\ds f''(t)}{\ds f(t)}+\left[\frac{\ds f'(t)}
{\ds f(t)}\right]^2\right]\eqskip
& & \hspace*{1cm}\times f^{n-1}(\alpha t)\psi^2(t)dt + \dint_{0}^{r}f^{n-1}(\alpha t)\left[\psi'(t)\right]^2 dt
\end{array}
\]
Plugging this back into~(\ref{eqn:raylquot}) yields
\[
\begin{array}{lll}
 \lambda(r)& \geqslant & \alpha^{2}\lambda(\alpha r)-\frac{\ds n-1}{\ds 2}
\frac{\ds \dint_{0}^{r}\left[\alpha^2(n-2)\left[\frac{\ds f'(\alpha t)}
{\ds f(\alpha t)}\right]^2 + \alpha^2\frac{\ds f''(\alpha t)}
{\ds f(\alpha t)}\right.}{\ds \dint_{0}^{r}f^{n-1}(\alpha t)\psi^2(t)dt}\eqskip
& & \hspace*{0.5cm}\frac{\ds \left.-\alpha(n-1)\frac{\ds f'(t)f'(\alpha t)}
{\ds f(t)f(\alpha t)}-\frac{\ds f''(t)}{\ds f(t)}+\left[\frac{\ds f'(t)}
{\ds f(t)}\right]^2\right]f^{n-1}(\alpha t)\psi^2(t)dt}{\ds \dint_{0}^{r}f^{n-1}(\alpha t)\psi^2(t)dt}.
\end{array}
\]
Consider first the case of $\alpha$ smaller than one. Then
\[
\begin{array}{lll}
 \frac{\ds \lambda(r) - \alpha^2 \lambda(\alpha r)}{\ds 1-\alpha} & \geqslant &
\frac{\ds n-1}{\ds 2(\alpha-1)}\frac{\ds \dint_{0}^{r}\left[\alpha^2(n-2)\left[\frac{\ds f'(\alpha t)}
{\ds f(\alpha t)}\right]^2 + \alpha^2\frac{\ds f''(\alpha t)}
{\ds f(\alpha t)}\right.}{\ds \dint_{0}^{r}f^{n-1}(\alpha t)\psi^2(t)dt}\eqskip
& & \hspace*{0.5cm}\frac{\ds \left.-\alpha(n-1)\frac{\ds f'(t)f'(\alpha t)}
{\ds f(t)f(\alpha t)}-\frac{\ds f''(t)}{\ds f(t)}+\left[\frac{\ds f'(t)}
{\ds f(t)}\right]^2\right]f^{n-1}(\alpha t)\psi^2(t)dt}{\ds \dint_{0}^{r}f^{n-1}(\alpha t)\psi^2(t)dt}
\end{array}
\]
and upon taking limits on both sides as $\alpha$ goes to $1^{-}$ we obtain, after some lengthy
computations,
\[
\begin{array}{lll}
  r \lambda'(r)+2\lambda(r) & \geqslant & \frac{\ds n-1}{\ds 2}\frac{\ds \dint_{0}^{r}\left[t g''(t)+
(n-1)t g(t)g'(t)+2g'(t)\right.}{\ds \dint_{0}^{r}f^{n-1}(t)\psi^2(t)dt}\eqskip
& & \hspace*{0.5cm}\frac{\ds\left.+(n-1)g^{2}(t)\right]f^{n-1}(t)\psi^{2}(t)dt}{\ds \dint_{0}^{r}f^{n-1}(t)\psi^2(t)dt},
\end{array}
\]
where $g(t)=f'(t)/f(t)$.

Now note that if we had taken $\alpha$ to be larger than one instead, then the inequalities
above would be reversed. Thus, after taking limits, now with $\alpha$ approaching one from
above, we would get the above inequality reversed. Hence this is an indentity, meaning we
have
\[
 \begin{array}{lll}
 r \lambda'(r)+2\lambda(r) & = & \frac{\ds n-1}{\ds 2}\frac{\ds \dint_{0}^{r}\left[t
h'(t)+2h(t)\right]f^{n-1}(t)\psi^{2}(t)dt}{\ds \dint_{0}^{r}f^{n-1}(t)\psi^2(t)dt},
\end{array}
\]
with $h(t)=g'(t)+(n-1)g^{2}(t)/2$.
Multiplying the above equation by $r$ and integrating between $r_{0}$ and $r$, yields the desired result.
\end{proof}

\subsection{Applications of Lemma~\ref{lem:basic1} to hyperbolic space and spheres\label{sec:appl}}

As in the classical Hadamard formula for the variation of the first eigenvalue with respect to domain
variations, the expressions in Lemma~\ref{lem:basic1} depend on the eigenfuntion of the unperturbed
domain. Although this may make it slightly difficult to use in general, there are situations for which,
depending on the behaviour of the function $H$, it might be possible to derive simplified expressions
for these formulae. The case of spaces of constant curvature is one such example as we shall now see.

\begin{thm}\label{thm:hypsph}
The first eigenvalue of a ball of radius $r$ in $\hyp^{n}$ $(0<r)$ or $\sph^{n}$ $(0<r<\pi)$ satisfies
\[
\begin{array}{ll}
\frac{\ds j_{0,1}^2}{\ds r^2}+\frac{\ds 1}{\ds 4}\left[\frac{\ds 1}{\ds r^2}-\frac{\ds 1}{\ds \s^{2}(r)}
\pm 1\right]\leqslant\lambda(r)\leqslant \frac{\ds j_{0,1}^2}{\ds r^2}\pm\frac{\ds 1}{\ds 3} & (n=2)\eqskip
\lambda(r)=\frac{\ds \pi^2}{\ds r^2}\pm 1 & (n=3)\eqskip
\frac{\ds j_{\frac{n-2}{2},1}^2}{\ds r^2}\pm\frac{\ds n(n-1)}{\ds 6}\leqslant\lambda(r)\leqslant
\frac{\ds j_{\frac{n-2}{2},1}^2}{\ds r^2}\pm\frac{\ds (n-1)^2}{\ds 4} & \eqskip
\hspace*{5cm}+\frac{\ds (n-1)(n-3)}{\ds 4}\left[\frac{\ds 1}{\ds \s^{2}(r)}-\frac{\ds 1}{\ds r^2}\right] & (4\leqslant n),
\end{array}
\]
where $\s$ denotes $\sinh$ or $\sin$, respectively, and in all indicated $\pm$ the plus and
minus signs are to be considered for $\hyp^{n}$ and $\sph^{n}$, respectively.
\end{thm}

\begin{rem}
In the case of the upper bounds in $\hyp^{n}$ $(n> 3)$, they approach $(n-1)^2/4$
as $r$ goes to infinity and are thus also asymptotically accurate in this limit.
The bounds for $\sph^{n}$ are not asymptotically accurate as $r$ approaches $\pi$ (except for $n=3$),
and indeed they may become quite poor in this limit.
\end{rem}

\begin{rem}
Note also that the above bounds are similar to those in Theorem 5.2 in~\cite{gage}. However,
the latter do not display the correct asymptotic behaviour as $r$ goes to zero.
\end{rem}
\begin{proof}
In the case of hyperbolic space with constant curvature $-1$, $f(t)=\sinh(t)$, and the function $H(t)= t h'(t)+2h(t)$ becomes
\[
H(t)=(n-1)\coth^{2}(t)-\frac{\ds 1}{\ds \sinh^2(t)}\left[2+(n-3)t \coth(t)\right].
\]
Its derivative may be written as
\[
H'(t)= \frac{\ds n-3}{\ds 2 \sinh^4(t)}\left[2t(2+\cosh(2t))-3\sinh(2t)\right],
\]
from which it follows that, for positive $t$, $H$ is decreasing for $n$ equal to two,
identically equal to $2$ for $n$ equal to three, and increasing for $n$ larger than three.
Using these monotonicity properties in Lemma~\ref{lem:basic1}, together with
\[
\lim_{t\to0}H(t)=\frac{\ds 2n}{\ds 3} \mbox{ and } \lim_{t\to\infty}H(t)=n-1,
\]
yields the bounds for $\hyp^{n}$ given in Theorem~\ref{thm:hypsph}.

Spheres with constant curvature $1$ correspond to $f(t)=\sin(t)$, and now
\[
H(t)=(n-1)\cot^{2}(t)-\frac{\ds 1}{\ds \sin^2(t)}\left[2+(n-3)t \cot(t)\right],
\]
while
\[
H'(t) = \frac{\ds n-3}{\ds 2 \sin^4(t)}\left[2t(2+\cos(2t))-3\sin(2t)\right].
\]
We again obtain that, for $0<t<\pi$, $H$ is decreasing when $n$ is two, constant ($-2$)
for $n$ equal to three and increasing for $n$ larger than $3$. This, together with
Lemma~\ref{lem:basic1}, yields the bounds for $\sph^{n}$ given in Theorem~\ref{thm:hypsph}.
\end{proof}

\section{Asymptotic expansion for small radius\label{asymptr0}}

In this section we consider the case of a small geodesic disk centred at the North Pole in the case where its
radius approaches zero. We assume $f$ to be sufficiently smooth for small $r$, and our aim
is to obtain the asymptotic expansion for $\l(r)$ and the associated eigenfunction $\psi$.

Our main result here is the following.

\begin{thm}\label{thmexpr0}
 The first eigenvalue of the geodesic disk centred at the North Pole $p$ and with radius $r$ satisfies
 \[
  \l(r) = \frac{\ds j_{\frac{n}{2}-1,1}^2}{\ds r^2}-\frac{\ds 1}{\ds 6}\mathcal{S}(p) + \left[ \alpha_{1}\mathcal{S}^{2}(p)+\alpha_{2}
\mathcal{S}''(p)\right]r^2 + \Odr(r^4),\quad
r\to0^+,
\]
where the coefficients $\alpha_{i},$ $i=1,2$ depend only on the dimension $n$ and are given by
\[
 \alpha_{1} = \frac{\ds c_{0}^2}{\ds 270 n^2(n-1)}\left[5\pi(n-1)(I_{2}+j_{\frac{n-2}{2},1} I_{3})-3(n+2)I_{1}\right]
\]
and
\[
 \alpha_{2}= -\frac{\ds c_{0}^2}{\ds 10}  I_{1},
\]
with
\[
 \begin{array}{l}
   c_{0} = \frac{\ds\sqrt{2}}{\ds\big|J_{\frac{n}{2}-1}'(j_{\frac{n}{2}-1,1})\big|}, \eqskip
  I_{1} = \dint\limits_{0}^{1} \xi^3 J_{\frac{n}{2}-1}^2(j_{\frac{n}{2}-1,1}\xi)\di\xi,    \eqskip
  I_{2} = \dint\limits_{0}^{1}\xi^3 J_{\frac{n}{2}-1}^3(j_{\frac{n}{2}-1,1}\xi) Y_{\frac{n}{2}-1}(j_{\frac{n}{2}-1,1}\xi)\di\xi,    \eqskip
  I_{3} = \dint\limits_{0}^{1} \xi^4  J_{\frac{n}{2}-1}^3(j_{\frac{n}{2}-1,1}\xi) Y_{\frac{n}{2}-1}'(j_{\frac{n}{2}-1,1}\xi)\di\xi.
 \end{array}
\]
\end{thm}
\begin{rem}
{\rm
The integrals $I_{k}$ ($k=2,3$) above are, in general, not computable in terms of known constants. However, in
dimension three it is possible to carry out all the computations explicitly to obtain that $\alpha_{1}$ vanishes in that case
and thus
\[
 \l(r) = \frac{\ds \pi^2}{\ds r^2}-\frac{\ds 1}{\ds 6}\mathcal{S}(p) + \left( \frac{\ds 1}{\ds 2\pi^2}-\frac{\ds 1}{\ds 3}\right)
 \frac{\ds \mathcal{S}''(p)}{\ds 10}r^2+ \Odr(r^4).
\]
}
\end{rem}

\begin{rem}
Although the above theorem provides only three terms in the asymptotics for the eigenvalue, our technique allows us to construct the complete asymptotic expansion, cf. equations (\ref{5.7}), (\ref{5.9}), (\ref{5.10}), (\ref{5.12}), (\ref{5.16}), (\ref{5.17}) below. And in the theorem we calculate explicitly in a convenient form only first three terms just to demonstrate the result.
\end{rem}

In order to handle the situation where $r\to0^+$,
it is natural to introduce a rescaled variable $\xi:=t r^{-1}$ and to rewrite the eigenvalue
problem (\ref{2.0}) as
\begin{align}\label{5.2}
-&\frac{d^2 \psi}{d\xi^2}-(n-1)r\frac{f'(r\xi)}{f(r\xi)}\frac{d \psi}{d\xi}=r^2\l \psi\quad \text{in}\quad (0,1),
\\
&\psi'(0,r)=\psi(1,r)=0.\label{5.5}
\end{align}
In view of (\ref{2.0}) we see that the coefficient at the first derivative behaves as
\begin{equation*}
r\frac{f'(r\xi)}{f(r\xi)}=\frac{1}{\xi}+r^2 f_0(r\xi),
\end{equation*}
where the function $f_0(t)$ is bounded uniformly in $[0,r_0]$. We substitute the last identity into (\ref{5.2}). It leads us to the equation
\begin{equation}
-\frac{d^2}{d\xi^2}-\frac{n-1}{\xi}\frac{d \psi}{d\xi}-(n-1)r^2f_0(r\xi) \frac{du}{d\xi}=r^2\l \psi\quad \text{in}\quad (0,1),
\label{5.4}
\end{equation}
with  boundary conditions (\ref{5.5}).

We introduce Hilbert spaces $Y_0$ and $Y_1$,
\begin{align*}
&Y_0:=\{u\in L_{2,loc}(0,1): \|u\|_{Y_0}<\infty\},\quad Y_1:=\{u\in \Hloc^1(0,1): \|u\|_{Y_1}<\infty\},
\\
&(u,v)_{Y_0}:=\int\limits_{0}^{1}\xi^{n-1}u\overline{v}\di\xi,\quad
(u,v)_{Y_1}^2:=(u',v')_{Y_0}+(u,v)_{Y_0}.
\end{align*}

We let
\begin{equation*}
\HH_*:=-\frac{d^2}{d\xi^2}-\frac{n-1}{\xi}\frac{d}{d\xi}\quad\text{on}\quad (0,1)
\end{equation*}
with the boundary conditions (\ref{5.5}). More precisely,
operator $\HH_*$ is understood as the associated one with the quadratic form
$h_0[u]:=\|u'\|_{Y_0}^2$ on $Y_1$. By $\Dom(\HH_*)$ we denote the domain of
operator $\HH_*$.

\begin{lemma}\label{lm5.1}
For each $u\in\Dom(\HH_*)$ the estimate
\begin{equation*}
\|f_0(r\,\cdot)u'\|_{Y_0}\leqslant C\big(\|\HH_* u\|_{Y_0}+\|u\|_{Y_0}\big)
\end{equation*}
holds true, where the constant $C$ is independent of $u$.
\end{lemma}
\begin{proof}
The statement of the lemma follows easily from the two estimates
\begin{equation*}
\begin{array}{lcl}
h_0[u]=(\HH_*u,u)_{Y_0}\leqslant \frac{1}{2} \big(\|\HH_* u\|_{Y_0}^2+\|u\|_{\HH_*}^2\big) &
 \mbox{ and } & \quad \|f_0(r\,\cdot) u'\|_{Y_0}\leqslant C h_0[u],
 \end{array}
\end{equation*}
valid for each $u\in\Dom(\HH_*)$.
\end{proof}

The last lemma implies that the operator $(n-1) f_0(r\xi)\frac{d}{d\xi}$ is $\HH_*$-bounded with the bounds independent of $r$. Hence, we can consider
problem (\ref{5.4}), (\ref{5.5}) as a small perturbation of the eigenvalue problem for $\HH_*$. Therefore,
eigenvalue $\l(r)$ and the
associated eigenfunction can be represented as the convergent series
\begin{equation}\label{5.7}
\l(r)=r^{-2}\sum\limits_{j=0}^{\infty} r^{2j} \l_j,\quad \psi(\xi,r)=\sum\limits_{j=0}^{\infty}r^{2j} \psi_j(\xi,r),
\end{equation}
where the latter converges in $Y_1$. We substitute these series into (\ref{5.4}) and (\ref{5.5}) to obtain the following equations for $\psi_j$,
\begin{equation}\label{5.9}
\HH_* \psi_0=\l_0 \psi_0 \mbox{ and } \quad
(\HH_*-\l_0) \psi_j=\l_j \psi_0+\sum\limits_{k=1}^{j-1} \l_k \psi_{j-k} + (n-1)f_0(r\,\cdot) \psi'_{j-1}.
\end{equation}
The ground state of $\HH_*$ is expressed in terms of the Bessel function of the first kind,
\begin{equation}\label{5.10}
\psi_0(\xi)=c_0 \xi^{-\frac{n}{2}+1}J_{\frac{n}{2}-1}(j_{\frac{n}{2}-1,1}\xi),\quad
\quad \l_0=j_{\frac{n}{2}-1,1}^2,
\end{equation}
where $c_0$ is the normalization constant given in Theorem~\ref{thmexpr0} ensuring
\begin{equation}\label{5.12}
\|\psi_0\|_{Y_0}=1.
\end{equation}

Consider
problem (\ref{5.9}) for $\psi_1$,
\begin{equation}\label{5.13}
(\HH_*-\l_0) \psi_1=\l_1 \psi_0+ (n-1)f_0(r\,\cdot) \psi_0'.
\end{equation}
The solvability condition to this problem is the orthogonality of the right hand side to $\psi_0$ in  $Y_0$. With (\ref{5.12}) taken into
consideration this determines $\l_1$ which is thus given by
\begin{equation}\label{5.14}
\l_1=-(n-1)(f_0(r\,\cdot) \psi_0',\psi_0)_{Y_0}.
\end{equation}
The solution to~(\ref{5.13}) exists and is defined up to an additive term $Cu_0$, which may be determined
by the orthogonality condition
\begin{equation}\label{5.15}
(\psi_1,\psi_0)_{Y_0}=0.
\end{equation}

In the same way we solve the succesion of
the problems for $\psi_j$, $j\geqslant 2$. The
corresponding equations~(\ref{5.9}) for $\psi_j$ are solvable,
if their right--hand sides are orthogonal to $\psi_0$ in $Y_0$. These solvability conditions then yield the formulae for $\l_j$,
\begin{equation}\label{5.16}
\l_j=-(n-1) (f_0(r\,\cdot)\psi'_{j-1},\psi_0)_{Y_0}.
\end{equation}
Here we have used that all the functions $\psi_j$ are defined modulo an additive term of the form $C \psi_0$, which are determined as above from the
orthogonality conditions
\begin{equation}\label{5.17}
(\psi_j,\psi_0)_{Y_0}=0.
\end{equation}
In this way all the coefficients of
series (\ref{5.9}) are uniquely determined. One can also check by
induction that these coefficients are bounded uniformly in $r$ (the coefficients $\psi_j$ in $Y_1$-norm).

The coefficients of series (\ref{5.7}) depend on $r$ through the coefficient $f_0(r\xi)$ in~(\ref{5.4}). On the other hand, if the
function $f$ is smooth enough, we can simplify the asymptotic expansion. More precisely, we can replace $f_0(r\xi)$ by its Taylor formula
as $r\to0^+$ and substitute it into~(\ref{5.9}),~(\ref{5.13}),~(\ref{5.14}),~(\ref{5.15}),
(\ref{5.16}) and~(\ref{5.17}), leading to another asymptotic expansion in terms of powers of $r$. We employ this fact to
construct leading terms of the asymptotics in a more explicit form than (\ref{5.7}). Namely, we assume $f\in C^5[0,r_0]$ and $f^{(2)}(0)=0$ with, as
mentioned in Section~\ref{notprelim}, the latter identity reflecting the smoothness of the manifold at the pole. Hence
\begin{equation}\label{5.21}
f_0(r\xi)= \frac{f'''(0)}{3}\xi + r^2\frac{3f^{(5)}(0)-5(f'''(0))^2}{90}\xi^3
+\Odr(r^3),\quad r\to0^+,
\end{equation}
uniformly in $\xi\in[0,1]$. We substitute this identity into~(\ref{5.14}),
\begin{equation}\label{5.18}
\begin{aligned}
\l_1(r)=&-\frac{(n-1)f'''(0)}{3} (\xi\psi_0',\psi_0)_{Y_0}
\\
&-r^2\frac{(n-1)\big(3f^{(5)}(0)-5(f'''(0))^2\big)}{90} (\xi^3\psi_0',\psi_0)_{Y_0}+\Odr(r^3).
\end{aligned}
\end{equation}
Integrating by parts and using~(\ref{5.12}) we get
\begin{equation}\label{5.19}
\begin{aligned}
& (\xi\psi_0',\psi_0)_{Y_0}=\int\limits_{0}^{1} \xi^n\psi_0'(\xi)\psi_0(\xi)\di\xi=-\frac{n}{2}
\int\limits_{0}^{1}\xi^{n-1}\psi_0^2(\xi)\di\xi=-\frac{n}{2},
\\
&(\xi^3\psi_0',\psi_0)_{Y_0}=\int\limits_{0}^{1} \xi^{n+2}\psi_0'(\xi)\psi_0(\xi)\di\xi=-\frac{n+2}{2} \int\limits_{0}^{1}\xi^{n+1}\psi_0^2(\xi)\di\xi
\\
&\hphantom{(\xi^2\psi_0',\psi_0)_{Y_0}}= - \frac{(n+2)c_0^2}{2} \int\limits_{0}^{1} s^3 J_{\frac{n}{2}-1}^2(j_{\frac{n}{2}-1,1}s)\di s.
\end{aligned}
\end{equation}
Then the right hand side in (\ref{5.13}) casts into the form
\begin{equation*}
\l_1(r)\psi_0(\xi)+(n-1)f_0(r\xi) \psi_0'(\xi) = \frac{(n-1)f'''(0)}{3} \left(\frac{n}{2}\psi_0(\xi)+\xi\psi_0'(\xi) \right) + \Odr(r^2),\quad r\to0^+,
\end{equation*}
uniformly in $\xi\in[0,1]$,
and the function $\psi_1$ may then be represented as
\begin{equation}\label{5.20}
\psi_1(\xi,r)=\frac{(n-1)f'''(0)}{3}\Psi_1(\xi,\r)+\Odr(r^2),\quad r\to0^+,
\end{equation}
uniformly in $\xi\in[0,1]$, where $\Psi_1$ is the solution to the equation
\begin{equation*}
(\HH_*-j_{\frac{n}{2}-1,1}^2) \psi_1=\frac{n}{2}\psi_0+\xi\psi_0'
\end{equation*}
satisfying the orthogonality condition $(\Psi_1,\psi_0)_{Y_0}=0$. The function $\Psi_1$ can be found explicitly,
\begin{align*}
&\Psi_1(\xi)=\frac{\pi c_0}{2}\xi^{-\frac{n}{2}+1} \widetilde{\Psi}_1(\xi),
\\
&\widetilde{\Psi}_1(\xi):=\frac{\xi^2}{2} J_{\frac{n}{2}-1}^2(j_{\frac{n}{2}-1,1}\xi) Y_{\frac{n}{2}-1}(j_{\frac{n}{2}-1,1}\xi) -
J_{\frac{n}{2}-1}(j_{\frac{n}{2}-1,1}\xi) \widehat{\Psi}_1(\xi)-c_1c_0^2 J_{\frac{n}{2}-1}(j_{\frac{n}{2}-1,1}\xi),
\\
&\widehat{\Psi}_1(\xi):=\int\limits_{0}^{\xi} s J_{\frac{n}{2}-1}(j_{\frac{n}{2}-1,1}s)(s Y_{\frac{n}{2}-1}(j_{\frac{n}{2}-1,1}s))'\di s,
\\
&c_1:=\int\limits_{0}^{1} \left(\frac{\xi^3}{2}J_{\frac{n}{2}-1}^3(j_{\frac{n}{2}-1,1}\xi)  Y_{\frac{n}{2}-1}(j_{\frac{n}{2}-1,1}\xi) - \xi
J_{\frac{n}{2}-1}
^2(j_{\frac{n}{2}-1,1}\xi) \widehat{\Psi}_1(\xi)\right)\di\xi,
\end{align*}
where $Y_{\frac{n}{2}-1}$ is the Bessel function of second kind. This identity together with~(\ref{5.20}),~(\ref{5.16}) and~(\ref{5.21}) yield
\begin{equation}\label{5.22}
\l_2(r)=-\frac{(n-1)^2(f'''(0))^2 c_0 }{9} \int\limits_{0}^{1} \xi^{\frac{n}{2}+1} J_{\frac{n}{2}-1}(j_{\frac{n}{2}-1,1}\xi)
\Psi_1'(\xi)\di\xi+\Odr(r^2).
\end{equation}
Integrating by parts, we obtain
\begin{align*}
\int\limits_{0}^{1}&\xi^{\frac{n}{2}+1} J_{\frac{n}{2}-1}(j_{\frac{n}{2}-1,1}\xi) \Psi'_1(\xi)\di\xi
\\
=&-\left(\frac{n}{2}+1\right) \int\limits_{0}^{1} \left(\frac{\xi^3}{2} J_{\frac{n}{2}-1}^3(j_{\frac{n}{2}-1,1}\xi)
Y_{\frac{n}{2}-1}(j_{\frac{n}{2}-1,1}\xi) - \xi J_{\frac{n}{2}-1}^2(j_{\frac{n}{2}-1,1}\xi) \widehat{\Psi}_1(\xi) \right)\di\xi
\\
&-j_{\frac{n}{2}-1,1}\int\limits_{0}^{1} \bigg(\frac{\xi^4}{2}J^2_{\frac{n}{2}-1}(j_{\frac{n}{2}-1,1}\xi) J'_{\frac{n}{2}-1}(j_{\frac{n}{2}-1,1}\xi)
Y_{\frac{n}{2}-1}(j_{\frac{n}{2}-1,1}\xi)
\\
&\hphantom{-j_{\frac{n}{2}-1,1}\int\limits_{0}^{1} \bigg(} -\xi^2
J_{\frac{n}{2}-1}(j_{\frac{n}{2}-1,1}\xi) J'_{\frac{n}{2}-1}(j_{\frac{n}{2}-1,1}\xi) \widehat{\Psi}_1(\xi)\bigg)\di\xi
\\
&+c_1c_0^2\int\limits_{0}^{1}\xi  J_{\frac{n}{2}-1}(j_{\frac{n}{2}-1,1}\xi)
\left(\left(\frac{n}{2}+1\right) J_{\frac{n}{2}-1}(j_{\frac{n}{2}-1,1}\xi) + j_{\frac{n}{2}-1,1}\xi
J'_{\frac{n}{2}-1}(j_{\frac{n}{2}-1,1}\xi)\right)\di\xi.
\end{align*}
Since
\begin{align*}
&-j_{\frac{n}{2}-1,1}\int\limits_{0}^{1} \frac{\xi^4}{2}J^2_{\frac{n}{2}-1}(j_{\frac{n}{2}-1,1}\xi) J'_{\frac{n}{2}-1}(j_{\frac{n}{2}-1,1}\xi)
Y_{\frac{n}{2}-1}(j_{\frac{n}{2}-1,1}\xi)\di\xi
\\
&\hphantom{j_{\frac{n}{2}-1,1}}=\frac{1}{6} \int\limits_{0}^{1}
J^3_{\frac{n}{2}-1}(j_{\frac{n}{2}-1,1}\xi)   \big( \xi^4
Y_{\frac{n}{2}-1}(j_{\frac{n}{2}-1,1}\xi)
\big)'\di\xi,
\\
&j_{\frac{n}{2}-1,1}\int\limits_{0}^{1} \xi^2 J_{\frac{n}{2}-1}(j_{\frac{n}{2}-1,1}\xi)
J'_{\frac{n}{2}-1}(j_{\frac{n}{2}-1,1}\xi)\widehat{\Psi}_1(\xi)\di\xi
\\
&\hphantom{j_{\frac{n}{2}-1,1}}= - \frac{1}{2} \int\limits_{0}^{1} \xi^3
J_{\frac{n}{2}-1}^3(j_{\frac{n}{2}-1,1}\xi)  \big( \xi
Y_{\frac{n}{2}-1}(j_{\frac{n}{2}-1,1}\xi)
\big)'
-\int\limits_{0}^{1} \xi J_{\frac{n}{2}-1}^2(j_{\frac{n}{2}-1,1}\xi) \widehat{\Psi}_1(\xi)\di\xi,
\\
&\int\limits_{0}^{1} \xi J_{\frac{n}{2}-1}^2(j_{\frac{n}{2}-1,1}\xi)\di\xi=\frac{1}{c_0^2},
\\
&\int\limits_{0}^{1}j_{\frac{n}{2}-1,1}\xi^2 J_{\frac{n}{2}-1}(j_{\frac{n}{2}-1,1}\xi)
J'_{\frac{n}{2}-1}(j_{\frac{n}{2}-1,1}\xi)\di\xi=-\int\limits_{0}^{1} \xi J_{\frac{n}{2}-1}^2(j_{\frac{n}{2}-1,1}\xi) \di\xi=-\frac{1}{c_0^2},
\end{align*}
we finally get
\begin{align*}
\int\limits_{0}^{1}\xi^{\frac{n}{2}+1} J_{\frac{n}{2}-1}(j_{\frac{n}{2}-1,1}\xi) \Psi'_1(\xi)\di\xi =& -\frac{\pi c_0}{6} \int\limits_{0}^{1}
\xi^3 J_{\frac{n}{2}-1}^3(j_{\frac{n}{2}-1,1}\xi) Y_{\frac{n}{2}-1}(j_{\frac{n}{2}-1,1}\xi)\di\xi
\\
&-\frac{\pi c_0}{6} \int\limits_{0}^{1} \xi^4  J_{\frac{n}{2}-1}^3(j_{\frac{n}{2}-1,1}\xi) Y_{\frac{n}{2}-1}'(j_{\frac{n}{2}-1,1}\xi)\di\xi.
\end{align*}
Formulae~(\ref{5.18}),~(\ref{5.19}),
and~(\ref{5.22}) yield the desired asymptotics for $\l(r)$,
\begin{align*}
&\l(r)=\frac{\ds j_{\frac{n}{2}-1,1}^{2}}{r^2}+\frac{n(n-1)f'''(0)}{6}+r^2\widetilde{\l}_2+\Odr(r^4),
\\
&\widetilde{\l}_2=\frac{(n-1)(n+2)\big(3f^{(5)}(0)-5(f'''(0))^2\big)c_0^2}{180}
\int\limits_{0}^{1} \xi^3 J_{\frac{n}{2}-1}^2(j_{\frac{n}{2}-1,1}\xi)\di\xi
\\
&\hphantom{\widetilde{\l}_2=}+\frac{\pi(n-1)^2(f'''(0))^2c_0^2}{54} \Bigg( \int\limits_{0}^{1}
\xi^3 J_{\frac{n}{2}-1}^3(j_{\frac{n}{2}-1,1}\xi) Y_{\frac{n}{2}-1}(j_{\frac{n}{2}-1,1}\xi)\di\xi
\\
&\hphantom{\widetilde{\l}_2=\Bigg(}
 + j_{\frac{n}{2}-1,1} \int\limits_{0}^{1} \xi^4  J_{\frac{n}{2}-1}^3(j_{\frac{n}{2}-1,1}\xi) Y_{\frac{n}{2}-1}'(j_{\frac{n}{2}-1,1}\xi)\di\xi\Bigg).
\end{align*}

>From the expression for the scalar curvature at a point $p$ given by~\eqref{exprscalarcurv}
we obtain that at the North Pole ($t=0$) this becomes $\mathcal{S}(p) = -n(n-1)f^{(3)}(0)$, while $\mathcal{S}'(p)=0$ and
\[
 \mathcal{S}''(p) = \frac{\ds (n+2)(n-1)}{\ds 6}\left[ \left(f^{(3)}(0)\right)^{2}-f^{(5)}(0)\right].
\]
Solving this for $f^{(3)}$ and $f^{(5)}$ and substituting above yields the expressions for $\alpha_{1}$ and $\alpha_{2}$ in Theorem~\ref{thmexpr0}.

\section{Asymptotic expansion for maximal radius\label{asymptrmax}}

In this section  we consider the case of a compact manifold with the function $f$ satisfying
\begin{equation}\label{5.1}
f(R)=0,\quad f'(R)=A<0, \quad f''(R)=0.
\end{equation}
In contrast to the previous section, here we treat the case of the disk being close to the whole manifold, namely, $r\to R^{-}$. Note that
since the disk is centred at the North pole, we have that the maximal radius $R$ equals the diameter of the manifold.

Our aim is to construct the asymptotic expansion for $\l(r)$ as $r\to R^{-}$. And the main result reads as follows.

\begin{thm}\label{th5.1}
The first eigenvalue of the geodesic disk centered at the North Pole $p$ and with radius $r$ satisfies as $r\to R^{-}$
\begin{equation}\label{5.45}
\begin{aligned}
\l(r)=&\frac{V(r)(2-n)}{B_{n,1}^{(1)}}(r-R)^{n-2} \Bigg(
1- \frac{B_{n,2}^{(1)}(2-n)}{B_{n,1}^{(1)}(4-n)}(r-R)^2
\\
&+ \Bigg( \Bigg(\frac{B_{n,2}^{(1)}(2-n)}{B_{n,1}^{(1)}(4-n)}\Bigg)^2
-\frac{B_{n,3}^{(1)}(2-n)}{B_{n,1}^{(1)}(6-n)}
\Bigg)(r-R)^4
\Bigg)
\\
&+\Odr\big((R-r)^{-n+3}\big),
\end{aligned}
\end{equation}
for $n\geqslant 7$,
\begin{equation}\label{5.46a}
\begin{aligned}
\l(r)=&-\frac{4 V(R)}{B_{6,1}^{(1)}}(r-R)^4 \Bigg[1-\frac{2 B_{6,2}^{(1)}}{B_{6,1}^{(1)}}(r-R)^2+ \frac{4 B_{6,3}^{(1)}}{B_{6,1}^{(1)}}(r-R)^4\ln(R-r)
\\
&+ \Bigg(
\frac{2B_{6,2}^{(1)}}{B_{6,1}^{(1)}}
\Bigg)^2(r-R)^4\Bigg]+\Odr\big((R-r)^9\big)
\end{aligned}
\end{equation}
for $n=6$,
\begin{equation}
\label{5.49}
\begin{aligned}
\l(r)=&-\frac{3V(R)}{B_{5,1}^{(1)}} (r-R)^3 + \frac{9V(R)B_{5,2}^{(1)}}{\big(B_{5,1}^{(1)}\big)^2}(r-R)^5
\\
& + \left(B_1^{(2)}-\frac{9 V(R) B_0^{(2)}}{\big(B_{5,1}^{(1)}\big)^2}
\right)(r-R)^6 + \Odr\big((R-r)^7\big)
\end{aligned}
\end{equation}
for $n=5$,
\begin{equation}
\begin{aligned}
\l(r)=&-\frac{2V(R)}{B_{4,1}^{(1)}} (r-R)^2 - \frac{4 V(R) B_{4,2}^{(1)}}{\big(B_{4,1}^{(1)}\big)^2}(r-R)^4\ln(R-r)
\\
&+ \left(B_3^{(2)}-\frac{4V(R)B_2^{(2)}}{\big(B_{4,1}^{(1)}\big)^2}\right) (r-R)^4+\Odr\big((R-r)^5\ln^2(R-r)\big)
\end{aligned}\label{5.52a}
\end{equation}
for $n=4$,
\begin{equation}
\begin{aligned}
\l(r)=&-\frac{V(R)}{B_{3,1}^{(1)}}(r-R)+ \left(B_6^{(2)}- \frac{V(R)B_4^{(2)}}{\big(B_{3,1}^{(1)}\big)^2}\right) (r-R)^2
 \\
& + (B_5^{(2)}+B_7^{(2)}+B_{10}^{(2)}) (r-R)^3+ \Odr\big((R-r)^4\big)
\end{aligned}\label{5.60a}
\end{equation}
for $n=3$,
\begin{equation}
\begin{aligned}
\l(r)=& \frac{V(R)}{B_{2,1}^{(1)}}\ln^{-1}(R-r) +\left(\frac{B_{12}^{(2)}V^2(R)}{\big(B_{2,1}^{(1)}\big)^3}- \frac{ B_{11}^{(2)}V(R)}{\big(B_{2,1}^{(1)}\big)^2}\right) \ln^{-2}(R-r)
\\
& + \left( \frac{B_{13}^{(2)} V^2(R)}{\big(B_{2,1}^{(1)}\big)^3}+ \frac{\big(B_{11}^{(2)}\big)^2 V(R)}{\big(B_{2,1}^{(1)}\big)^3}- 2 \frac{\big(B_{12}^{(2)}\big)^2V^4(R)} {\big(B_{2,1}^{(1)}\big)^6}
 \right.
\\
 &\left.- \frac{B_{11}^{(2)} B_{12}^{(2)} V^2(R)}{\big( B_{2,1}^{(1)}\big)^4}+B_{14}^{(2)}
\right)\ln^{-3}(R-r)
+\Odr\big(\ln^{-4}(R-r)\big)
\end{aligned}\label{5.65}
\end{equation}
for $n=2$. Here the constants are given by the formulae
\begin{align*}
& B_{n,1}^{(1)}:=\frac{V^2(R)}{\om_{n-1} A^{n-1}},\quad B_{n,2}^{(1)}:=-\frac{A_2 V^2(R)}{\om_{n-1}A^{n-1}},\quad B_{n,3}^{(1)}:=\frac{(A_2^2-A_4)V^2(R)}{\om_{n-1}A^{n-1}},
\\
& A_2:=\frac{(n-1)f'''(0)}{3! A^{n-1}},
\quad
A_4:=(n-1)\left( \frac{f^{(5)}(0)}{5!} + \frac{(n-2)\big(f'''(0)\big)^2}{12}
\right),
\\
&B_0^{(2)}:=\int\limits_{0}^{R} \left(\frac{V^2(t)}{V'(t)}-
B_{5,1}^{(1)}(t-R)^{-4} - B_{5,2}^{(1)} (t-R)^{-2}\right)\di t - \frac{B_{5,1}^{(1)}}{3R^3}-\frac{B_{5,2}^{(1)}}{R},
\\
&B_1^{(2)}:=-\frac{9 V^3(R)}{\om_4\big(B_{5,1}^{(1)}\big)^3 A^4} \int\limits_{0}^{R} \frac{V(s)-V(R)}{V(s)}V(s)\di s,
\\
&B_2^{(2)}:=\int\limits_{0}^{R} \left(\frac{V^2(t)}{V'(t)}- B_{4,1}^{(1)}(t-R)^{-3}-B_{4,2}^{(1)}(t-R)^{-1}\right)\di t + \frac{B_{4,1}^{(1)}}{2R^2}-B_{4,2}^{(1)}\ln R,
\\
&B_3^{(2)}:=\int\limits_{0}^{R}\frac{V(s)-V(R)}{V'(s)}V(s)\di s,
\\
&B_4^{(2)}:=\int\limits_{0}^{R} \left(\frac{V^2(t)}{V'(t)}- B_{3,1}^{(1)}(t-R)^{-2}\right)\di t-B_{3,1}^{(1)}R^{-1},
\\
&B_5^{(2)}:=\frac{\om_2 A^2\big(B_{-2}^{(1)}\big)^3-3 V(R) B_{3,2}^{(1)}B_{3,1}^{(1)}-3 V(R)\big(B_4^{(2)}\big)^2}{3\big(B_{3,1}^{(1)}\big)^3},
\\
&B_6^{(2)}:=-\frac{B_8^{(2)}V^2(R)}{\big(B_{3,1}^{(1)}\big)^3},
\quad B_8^{(2)}:=\frac{V(R)}{\om_2 A^2}\int\limits_{0}^{R} \frac{V(s)-V(R)}{V'(s)}V(s)\di s,
\\
&B_7^{(2)}:=-\frac{2V^2(R)B_4^{(2)}B_8^{(2)}}{\big(B_{3,1}^{(1)}\big)^4} -\left(\frac{B_9^{(2)}}{B_{3,1}^{(1)}}+ \frac{B_4^{(2)}B_8^{(2)}}{\big(B_{3,1}^{(1)}\big)^2}
\right)\frac{V^2(R)}{\big(B_{3,1}^{(1)}\big)^2},
\\
&B_9^{(2)}:=\int\limits_{0}^{R} \frac{V(t)}{V'(t)} \left( \int\limits_{0}^{t} \frac{V(s)-V(t)}{V'(s)}V(s)\di s-B_8^{(2)}(t-R)^{-2}
\right)\di t + \frac{B_8^{(2)}}{R},
\\
&B_{10}^{(2)}:= \frac{V^3(R)}{\om_2 \big( B_{3,1}^{(1)}
\big)^3 A^2} \int\limits_{0}^{R} \di s \frac{V(s)-V(R)}{V'(s)} \int\limits_{0}^{s} \frac{V(z)-V(s)}{V'(z)}V(z)\di z,
\\
& B_{11}^{(2)}:=-B_{2,1}^{(1)}\ln R + \int\limits_{0}^{R} \left(\frac{V^2(t)}{V'(t)}-B_{2,1}^{(1)}(t-R)^{-1}\right)\di t,
\\
&B_{12}^{(2)}:=-\frac{V(R)}{\om_1 A} \int\limits_{0}^{R} \frac{V(s)-V(R)}{V'(s)} V(s)\di s,
\\
&B_{13}^{(2)}:=-B_{12}^{(2)} \ln R + \int\limits_{0}^{R} \left(\frac{V(t)}{V'(t)} \int\limits_{0}^{t} \frac{V(s)-V(t)}{V'(s)}\di s- B_{12}^{(2)}(t-R)^{-1}\right)\di t,
\\
&B_{14}^{(2)}:=\frac{V^4(R)}{\om_2 A^2\big(B_{2,1}^{(1)}\big)^4}\int\limits_{0}^{R}\di s \frac{V(s)-V(R)}{V'(s)}\int\limits_{0}^{s} \frac{V(z)-V(s)}{V'(z)}V(z)\di z.
\end{align*}

\end{thm}

\begin{rem}
The first term in the asymptotics in the case of domains with a small hole was investigated by many authors -- see~\cite{fluc} and the references therein, and also~\cite{chav} and~\cite{manapl}, for instance; the latter of these includes the full expansion in the case of a two-dimensional manifold
with a small hole. With our approach we are able to obtain the complete asymptotic expansions in any dimension, cf.
identities (\ref{6.12}) below. As we see, for dimensions $2, 4$ and $6$ logarithmic terms appear in these expansions. While for $n=2$ this is quite natural as these terms are produced directly by the singularity of the Green function for such two-dimensional elliptic operators, for other
dimensions their appearance is not so evident and, to our knowledge, was not known before. Moreover, provided $f$ is smooth enough, say, $f$ is infinitely
differentiable, it can be shown that the complete asymptotic expansion for $\l(r)$ involves logarithmic terms in all even dimensions, see Remark~\ref{rm5.1}.
\end{rem}

First we construct the asymptotics formally and then we rigorously estimate the error terms. Since for $r=R$ the lowest eigenvalue of Laplace-Beltrami operator on the manifold is $0$ and the associated eigenfunction
is constant, we could assume that the leading term in the asymptotic expansion for $\psi(t,r)$ should be constant. On the other hand, the constant
function does not satisfy the boundary condition on $t=r$ in (\ref{eqn:basic}). The usual way to achieve the desired boundary condition is to employ
the boundary layer method \cite{VL} or the matching of asymptotic expansion \cite{Il}. Here we do not go in this way since it is possible to include
the inner expansion into the external one and to construct the full asymptotics as a series in terms of the variable $t$ without introducing the
rescaled variable. In order to do it, we have to take the leading term in the expansion for $\psi$ in a special form. Namely, we assume the
following ans\"atzes,
\begin{gather}\label{6.12}
\l(r)=\sum\limits_{j=1}^{\infty}\mu_n^j(r) \l_j(r),\quad \psi(t,r)=\sum\limits_{j=0}^{\infty} \mu_n^j(r)\psi_j(t,r),
\\
\mu_n(r):=\left\{
\begin{aligned}
&\ln^{-1}(R-r), && n=2,
\\
& (2-n)(r-R)^{n-2}, && n\geqslant 3.
\end{aligned}\right.  \nonumber
\end{gather}
where $\l_j$ and $\psi_j$ are to be determined. We define the function $\psi_0$ as the solution to the boundary value problem
\begin{equation*}
\left\{
\begin{aligned}
-&(f^{n-1}\psi_0')'=\mu \l_0 f^{n-1}\quad\text{in}\quad (0,r),
\\
&\psi_0'(0,r)=0,\quad \psi_0(0,r)=1,
\end{aligned}
\right.
\end{equation*}
 given by the formula
\begin{equation}\label{6.30}
\psi_0(t,r)=\l_0(r)
 \mu_n(r) \int\limits_{0}^{t} \frac{V(s)}{V'(s)}\di s,\quad \l_0
 (r):=\left(\mu_n(r)\int\limits_{0}^{r} \frac{V(s)}{V'(s)}\di s\right)^{-1}.
\end{equation}
It follows from (\ref{5.1}) that
\begin{align}\label{6.9}
&f(t)=A(t-R)+\odr(t-R),\qquad\qquad\qquad t\to R^{-}.
\\
&
\begin{aligned}
&\mu_n(r)\int\limits_{0}^{r} \frac{V(s)}{V'(s)}\di s=\frac{V(R)}{\om_{n-1}A^{n-1}}+\odr(1),\quad
r\to R^{-},
\\
&\l_0(r)=\frac{\om_{n-1}A^{n-1}}{V(R)}(1+\odr(1)), \qquad\quad\qquad
r\to R^{-},
\end{aligned}
\label{6.10}
\end{align}
It is also obvious that $\psi_0\in C^2[0,R]$ and
\begin{equation*}
\|\psi_0\|_{C[0,R]}\leqslant C,
\end{equation*}
where $C$ is a constant independent of $r$.

We plug in series (\ref{6.12}) into the eigenvalue problem (\ref{eqn:basic}) and equate the coefficients of like powers of $\mu$,
leading us to the boundary value problems for $\psi_j$,
\begin{equation}\label{6.13}
\begin{aligned}
-&(f^{n-1}\psi_1')'=\l_1 f^{n-1}\psi_0-\l_0 f^{n-1} &&\text{in}\quad (0,r), &&
\\
-&(f^{n-1}\psi_j')'=\l_j f^{n-1}\psi_0+ f^{n-1}\sum\limits_{k=1}^{j-1} \l_k \psi_{j-k} && \text{in}\quad (0,r), && j\geqslant 2,
\\
&\psi_j'(0,r)=\psi_j(r,r)=0, && &&  j\geqslant 1.
\end{aligned}
\end{equation}

\begin{lemma}\label{lm7.3}
Let $g\in C[0,R]$ and
\begin{equation}\label{6.31}
\int\limits_{0}^{r}f^{n-1}(t)g(t)\di t=0.
\end{equation}
Then the boundary value problem
\begin{equation}\label{6.1}
-(f^{n-1}u')'=f^{n-1}g\quad \text{in}\quad (0,r)
\qquad u'(0)=u(r)=0,
\end{equation}
has the unique solution given by the formula
\begin{equation}\label{6.7}
u(t)=\LL[g](t),\quad \LL[g](t):=\int\limits_{t}^{r} f^{-n+1}(s)\left(\int\limits_{0}^{s} f^{n-1}(z)g(z)\di z\right)\di s.
\end{equation}
It belongs to $C^2[0,r]$ and satisfies the uniform in $r$ estimate
\begin{equation}\label{6.32}
\|u\|_{C^1[0,R]}\leqslant C\|h\|_{C[0,R]}.
\end{equation}
\end{lemma}

\begin{proof}
It is clear that the function $u$ defined by (\ref{6.7}) solves (\ref{6.1}) and belongs to $C^2[0,R]$. Let us prove  estimate (\ref{6.32}). Due to
(\ref{6.31}) we have
\begin{equation*}
u'(t)=-f^{-n+1}(t)\int\limits_{0}^{t}f^{n-1}(s)g(s)\di s=
f^{-n+1}(t)\int\limits_{t}^{r} f^{n-1}(s) g(s)\di s.
\end{equation*}
By the boundedness of $g$ and the positiveness of $f$ these formulae imply
\begin{align*}
&|u'(t)|\leqslant \frac{V(t)}{V'(t)}\|g\|_{C[0,R]},
\\
&|u'(t)|\leqslant \frac{V(r)-V(t)}{V'(t)}\|g\|_{C[0,R]}\leqslant \frac{V(R)-
V(t)}{V'(t)}\|g\|_{C[0,R]}.
\end{align*}
Hence,
\begin{equation*}
|u'(t)|\leqslant \frac{\min\{V(t),
V(R)-V(t)\}}{V'(t)}\|g\|_{C[0,R]}\leqslant C\|g\|_{C[0,R]},
\end{equation*}
where the constant $C$ is independent of $r$ and $g$. Employing the last estimate and the identity
\begin{equation*}
u(t)=-\int\limits_{t}^{r} u'(s)\di s,
\end{equation*}
we arrive at (\ref{6.32}).
\end{proof}

We employ the last lemma to solve the problems (\ref{6.13}). In order for the series (\ref{6.12}) to be asymptotic the coefficients $\l_j$ and
$\psi_j$ should be bounded uniformly in $r$. Hence, to have the function $\psi_1$ bounded, the right hand side of the equation for $\psi_1$ should
satisfy (\ref{6.31}). It implies the formula for $\l_1$,
\begin{equation*}
\l_1
(r)=\frac{\l_0
 (r)V(r)}{\int\limits_{0}^{r} V'(s)\psi_0(s,r)\di s}.
\end{equation*}
We shall now compute the denominator in the last formula. To this end, integrate by parts taking into consideration the definition of $\psi_0$,
\begin{equation}\label{6.45}
\int\limits_{0}^{r} V'(t)\psi_0(t,r)\di t=\l_0\mu G(r),\quad G(t):=\int\limits_{0}^{t} \frac{V^2(s)}{V'(s)}\di s.
\end{equation}
Together with the definition of $\l_0$ in (\ref{6.30}), this allows us to rewrite the formula for $\l_1$,
\begin{equation}\label{6.40}
\l_1(r)=\frac{V(r)}{\mu_n(r) G(r)}.
\end{equation}
Again by (\ref{6.9}) we see that
\begin{equation*}
\l_1(r)=\frac{\om_{n-1}A^{n-1}}{V(R)}+\odr(\mu).
\end{equation*}

Since condition (\ref{6.31}) is satisfied,  the solution to the equation for $\psi_1$ in (\ref{6.13}) is given by the identity
\begin{equation*}
\psi_1=\LL[\l_1 f^{n-1}\psi_0-\l_0 f^{n-1}]\in C^2[0,R]
\end{equation*}
and this function is bounded uniformly in $r$, \begin{equation*}
\|\psi_1\|_{C^1[0,R]}\leqslant C.
\end{equation*}

In the same way we solve problems (\ref{6.13}) for $j\geqslant 2$. We first write condition (\ref{6.31}) that determines $\l_j$,
\begin{equation}\label{6.44}
\begin{aligned}
\l_j=&-\frac{1}{\int\limits_{0}^{r} f^{n-1}\psi_0\di t} \sum\limits_{k=1}^{j-1}\l_k \int\limits_{0}^{r} f^{n-1}\psi_{j-k}\di t
\\
=&-\frac{1}{\l_0\mu G
} \sum\limits_{k=1}^{j-1}\l_k \int\limits_{0}^{r} V'\psi_{j-k}\di t,
\end{aligned}
\end{equation}
where we have used (\ref{6.45}). Provided the functions $\psi_k$ and $\l_k$, $k\geqslant j-1$, are bounded uniformly in $r$ (the former in
the $C^1[0,R]$-norm), by (\ref{6.9}) we obtain that $\l_j$ is also bounded uniformly in $r$. Then, by Lemma~\ref{lm7.3}, the function $\psi_j$
reads as follows,
\begin{equation}\label{6.46}
\psi_j=\LL\big[\l_j \psi_0+ \sum\limits_{k=1}^{j-1} \l_k \psi_{j-k}\big].
\end{equation}
It belongs to $C^2[0,R]$ and is bounded uniformly in $r$ in the $C^1[0,R]$-norm.

In conclusion to the formal constructing we prove that  series (\ref{6.12}) are formal asymptotic solutions to (\ref{eqn:basic}). For $N\geqslant
0$ we let
\begin{equation}\label{6.21}
\l^{(N)}(r):=\sum\limits_{j=1}^{N} \mu_n^j(r)\l_j(r),\quad \psi^{(N)}(t,r):=\sum\limits_{j=0}^{N} \mu_n^j(r)\psi_j(t,r).
\end{equation}

\begin{lemma}\label{lm7.1}
Given any $N\geqslant 0$, for the functions $\l^{(N)}$ and $\psi^{(N)}$  the convergences
\begin{equation}\label{6.22}
\l^{(N)}\to0,\quad \|\psi^{(N)}-\psi_0\|_{C^2[0,r]}\to0,\quad \e\to0^+,
\end{equation}
and the equation
\begin{equation}\label{6.23}
(\HH_r-\l^{(N)})\psi^{(N)}=h^{(N)}
\end{equation}
hold true. The function $h_N\in C[0,r]$ satisfies the estimate
\begin{equation}\label{6.24}
\|h_N\|_{C[0,r]}\leqslant C_N \mu_n^{N+1},
\end{equation}
where $C_N$ is a constant independent of $\mu_n$ and $\e$.
\end{lemma}

\begin{proof}
The convergences (\ref{6.22}) follow directly from the uniform boundedness of $\l_j$ and $\|\psi_j\|_{C^1[0,R]}$ in $r$.

Employing  boundary value problems (\ref{6.13}) for $\psi_j$, by direct calculations we check that
\begin{equation*}
h^{(0)}=\l_0\mu_n\psi_0,\quad h^{(N)}=\sum\limits_{\genfrac{}{}{0 pt}{}{1\leqslant k,j\leqslant N}{k+j\geqslant N+1}} \mu_n^{k+j}\l_k \psi_j,\quad
N\geqslant 1.
\end{equation*}
satisfy the boundary value problem (\ref{6.23}).
Estimate (\ref{6.24}) follows directly from the last identity and the aforementioned boundedness of $\l_j$ and $\psi_j$.
\end{proof}

We proceed to the justification of the asymptotics. We first prove two auxiliary lemmas characterizing $\l(r)$.

\begin{lemma}\label{lm7.2}
The eigenvalue $\l(r)$ is the only one of the problem (\ref{eqn:basic}) which converges to zero as $\e\to0^+$. It is simple and satisfies the
estimate
\begin{align}\label{6.47}
&\l(r)\leqslant  \frac{\mu_n(r)\l_1(r)}{1+\mu_n\l_1(\mu_n) G^{-1}(r)
\int\limits_{0}^{r}\frac{G^2(t)}{G'(t)}\di t},
\\
&0\leqslant G^{-1}(r)
\int\limits_{0}^{r}\frac{G^2(t)}{G'(t)}\di t \leqslant \left\{
\begin{aligned}
&C\mu_n(r), && n=2,3,
\\
&C\mu_n(r)\ln\mu_n(r), && n=4,
\\
&C\mu_n^{\frac{2}{n-2}}(r), && n\geqslant 5,
\end{aligned}
\right.
\end{align}
where $C$ is a constant independent of $\mu_n$.
\end{lemma}

\begin{proof}
We first prove the upper bound for $\l(r)$. Using $\psi_0$ as a test function in the Rayleigh quotient (\ref{eqn:variational}),
we obtain
\begin{equation}\label{6.48}
\l(r)\leqslant \frac{\int\limits_{0}^{r} f^{n-1}(\psi_0')^2\di t}{\int\limits_{0}^{r} f^{n-1}\psi_0^2\di t}=\frac{G(r)}{\int\limits_{0}^{r} V'(t)
\left(\int\limits_{t}^{r} \frac{V(s)}{V'(s)}\di s\right)^2\di t}.
\end{equation}
The denominator may be simplified by integration by parts as follows
\begin{align*}
\int\limits_{0}^{r} V'(t)\left(\int\limits_{t}^{r} \frac{V(s)}{V'(s)}\di s\right)^2\di t= & 2
\int\limits_{0}^{r} G'(t)\int\limits_{t}^{r} \frac{V(s)}{V'(s)}\di s \di t= 2
\int\limits_{0}^{r} \frac{G(t) V(t)}{V'(t)}\di s \di t
\\
=& \frac{G^2(r)}{V(r)} + \int\limits_{0}^{r} \frac{V'(t) G^2(t)}{V^2(t)}\di t.
\end{align*}
Substituting this identity into (\ref{6.48}), we arrive at the first estimate in (\ref{6.47}). Let us prove the second one.

By (\ref{6.9}) we have
\begin{equation*}
G(r)=\frac{V^2(r)}{\mu_n}(1+\odr(1)).
\end{equation*}
It follows from the definition of $G$ and (\ref{6.9}) that
\begin{align*}
&G^{-1}(r)\int\limits_{0}^{r} \frac{G^2(t)}{G'(t)}\di t=G^{-1}(r) \left( \int\limits_{0}^{R/2}+\int\limits_{R/2}^{r}
\right)\frac{V'(t) G^2(t)}{V(t)}\di t
\\
&\hphantom{G^{-1}(r)}= G^{-1}(r) \int\limits_{R/2}^{r}\frac{V'(t) G^2(t)}{V(t)}\di t +C\mu_n(r)
\leqslant C\mu_n(r) \left( \int\limits_{R/2}^{r} f^{n-1}(t) G^2(t)\di t +1\right),
\\
&\mu_n(r)  \int\limits_{R/2}^{r} f^{n-1}(t) G^2(t)\di t\leqslant C\mu_n(r) \int\limits_{R/2}^{r} f^{n-1}(t)\left(\int\limits_{R/2}^{t} \frac{\di
s}{f^{n-1}(s)}\right)^2\di t
\\
&\hphantom{\mu_n(r)  \int\limits_{R/2}^{r} f^{n-1}(t) G^2(t)\di t}\leqslant \left\{
\begin{aligned}
&C\mu_n(r), && n=2,3,
\\
&C\mu_n(r)\ln\mu_n(r), && n=4,
\\
&C\mu_n^{\frac{2}{n-2}}(r), && n\geqslant 5,
\end{aligned}
\right.
\end{align*}
where $C$ denotes various inessential constants independent of $\mu_n$. The second estimate in (\ref{6.47}) is proven.

Estimate (\ref{6.47}) yields that $\l(r)\to0^+$ as  $r\to R^{-}$. It remains to prove that there are no other eigenvalues of (\ref{eqn:basic})
converging to zero.

Let $\l_*(r)$ be an eigenvalue of (\ref{eqn:basic}) converging to zero and $\psi_*(t,r)$ be an associated eigenfunction. We normalize $\psi_*$ by the
condition
\begin{equation}\label{6.50}
\max\limits_{[0,r]} |\psi_*(\cdot,r)|=1.
\end{equation}
Then we represent $\psi_*$ as
\begin{equation*}
\psi_*(t,r)=a_*(R)+\widetilde{\psi}_*(t,r),\quad a_*(r):=r^{-1}\int\limits_{0}^{r}\psi_*\di t,\quad\int\limits_{0}^{r} f^{n-1} \widetilde{\psi}_*\di
t=0.
\end{equation*}
We observe that $a_*$ and $\widetilde{\psi}_*$ are bounded uniformly in $t$ and $r$. Hence, we can apply Lemma~\ref{lm7.3} to
$g=\l_*\widetilde{\psi}_*$ and represent $\psi_*$ as
\begin{equation}\label{6.52}
\psi_*=\frac{\l_* a_*}{ \l_0\mu_n}\psi_0+\widehat{\psi}_*,\quad \|\widehat{\psi}_*\|_{Y_0}\leqslant C\l_*,
\end{equation}
where $C$ is a constant independent of $r$.  Now
we employ  normalization (\ref{6.50}),
\begin{align*}
1\geqslant  |\psi_*(0,r)|=\left|\frac{\l_* a_*}{ \l_0\mu_n}\psi_0(0,r)+\widehat{\psi}_*(0,r)
\right|\geqslant \frac{\l_*a_*}{\l_0\mu_n}-C\l_*
\end{align*}
that implies
\begin{equation*}
\l_* a_*\leqslant C_1\l_0\mu_n,
\end{equation*}
where $C_1$ is a constant independent of $\e$. At the same time, \begin{equation*}
1=|\psi_*(t_0,r)|\leqslant  \frac{\l_* a_*}{ \l_0\mu_n}\psi_0(0,r)+C\l_*=\frac{\l_* a_*}{ \l_0\mu_n}+C\l_*
\end{equation*}
and therefore
\begin{equation*}
\l_*a_*\geqslant C_2\l_0\mu_n,
\end{equation*}
where $C_2$ is a constant independent of $r$. Hence, the first term in the right hand side of (\ref{6.52}) is of order $\Odr(1)$ while
$\widetilde{\psi}_*$ is of order $\Odr(r)$. If we assume now that there are two eigenvalues of (\ref{eqn:basic}) converging to zero, then the
associated eigenfunctions satisfy (\ref{6.52}). At the same time, it contradicts to the fact that these eigenfunctions should be linear independent.
\end{proof}

By the proven lemma the closest to $\l^{(N)}$ eigenvalue of $\HH_r$ is $\l(r)$. Hence,
\begin{equation*}
\|(\HH_r-\l^{(N)})^{-1}\|=\frac{1}{|\l^{(N)}(r)-\l(r)|},
\end{equation*}
and by Lemma~\ref{lm7.1} it follows
\begin{equation}\label{6.57}
\begin{aligned}
&\|\psi^{(N)}\|_{X_r^0}\leqslant \|(\HH_r-\l^{(N)})^{-1}\|\, \|h^{(N)}\|_{X_r^0}\leqslant \frac{\|h^{(N)}\|_{X_r^0}}{|\l^{(N)}(r)-\l(r)|}
\\
&\hphantom{\|\psi^{(N)}\|_{X_r^0}}\leqslant \frac{C_N R}{|\l^{(N)}(r)-\l(r)|},
\\
&\|\psi^{(N)}\|_{X_r^0}=\|\psi_0\|_{X_r^0}+\odr(1).
\end{aligned}
\end{equation}
We calculate the norm $\|\psi_0\|_{X_r^0}$ by integration by parts and employing (\ref{6.10}),
\begin{align*}
\|\psi_0\|_{X_r^0}^2=&\int\limits_{0}^{r} f^{n-1}(t) \psi_0^2(t,r)\di t=\frac{2\l_0^2\mu_n^2}{\om_{n-1}}\int\limits_{0}^{r} \frac{V^2(t)}{V'(t)} \left(
\int\limits_{t}^{r} \frac{V(s)}{V'(s)}\di s\right)\di t
\\
=& \frac{2\l_0^2\mu_n^2 V^3(r)}{\om_{n-1}} \int\limits_{R/2}^{r} \frac{1}{V'(t)} \left(
\int\limits_{R/2}^{t} \frac{\di s}{V'(s)}\right)\di t \cdot (1+\odr(1))=\frac{2 V(R)}{\om_{n-1}} (1+\odr(1)).
\end{align*}
We substitute the obtained identity into (\ref{6.57}),
\begin{equation}\label{6.53}
\frac{2 V(R)}{\om_{n-1}} (1+\odr(1)) \leqslant \frac{C_N R \mu_n^{N+1}}{|\l(r)-\l^{(N)}(r)},\quad |\l(r)-\l^{(N)}(r)|=\Odr(\mu_n^{N+1}),
\end{equation}
that justifies  asymptotics (\ref{6.12}) for $\l(r)$.

Let us justify  asymptotics (\ref{6.12}) for $\psi$. By \cite[Ch. V, Sec. 3.5, Eq. (3.21)]{K} we have the representation
\begin{equation}\label{6.54}
\psi^{(N)}=\frac{(h^{(N)},\psi)_{X_r^0}}{\l(r)-\l^{(N)}(r)} \psi + \RR(r) h^{(N)},
\end{equation}
where $\RR(r)$ is an operator in $X_r^0$ bounded uniformly in $r$ and mapping $X_r^0$ into the orthogonal complement of $\psi$ in $X_r^0$, and $\psi$
is supposed to be normalized in $X_r^0$. Therefore,
\begin{equation}\label{6.55}
(\psi^{(N)},\psi)_{X_r^0}=\frac{(h^{(N)},\psi)_{X_r^0}}{\l(r)-\l^{(N)}(r)},\quad \|\RR(r)h^{(N)}\|_{X_r^0}=\Odr(\mu_n^{N+1}).
\end{equation}
It follows from definition (\ref{6.21}) of $\psi^{(N)}$  and the boundedness of $\|\psi_j\|_{C[0,R]}$ that for each $N\geqslant 0$
\begin{equation*}
 \frac{(h^{(N)},\psi)_{X_r^0}}{\l(r)- \l^{(N)}(r)}-\frac{(h^{(N+1)},\psi)_{X_r^0}}{\l(r)-\l^{(N+1)}(r)} =\Odr(\mu_n^{N+1}).
\end{equation*}
Hence, there exists a function $b(\mu_n)$ such that for each $N\geqslant 0$
\begin{equation*}
\frac{(h^{(N)},\psi)_{X_r^0}}{\l(r)-\l^{(N)}(r)}=b(\mu_n)+\Odr(\mu_n^{N+1}). \end{equation*}
We substitute this identity and the second relation from (\ref{6.55}) into (\ref{6.54}),
\begin{equation*}
b(\mu_n)\psi=\psi^{(N)}+\Odr(\mu_n^{N+1})
\end{equation*}
in the norm of $X_r^0$. This identity is also valid in the norm of $X_r$, since by (\ref{6.53}) the equations for $\psi$ and $\psi^{(N)}$
\begin{align*}
&\HH_r(\psi^{(N)}-\psi)=\l^{(N)}(\psi^{(N)}-\psi)+(\l^{(N)}-\l)\psi+h^{(N)},
\\
&\|(\psi^{(N)}-\psi)'\|_{X_r^0}=\l^{(N)}\|\psi^{(N)}-\psi\|_{X_r^0}^2 + (\l^{(N)}-\l)(\psi,\psi^{(N)}-\psi)_{X_r^0}
\\
&\hphantom{\|(\psi^{(N)}-\psi)'\|_{X_r^0}=} + (h^{(N)},\psi^{(N)}-\psi)_{X_r^0}=\Odr(\mu_n^{2N+2}).
\end{align*}
The justification is complete.

As in the previous section, let us calculate the leading terms of asymptotics (\ref{6.12}) in a more explicit form. We assume that $f\in C^6[0,R]$ and
$f^{(4)}(0)=f''(0)=0$. Then
\begin{align*}
&
\begin{aligned}
f(t)=&A(t-R) + \frac{f'''(0)}{3!}(t-R)^3
\\
&+ \frac{f^{(5)}(0)}{5!} (t-R)^5+ \Odr\big((t-R)^6\big),\quad t\to R^{-},
\end{aligned}
\\
&
\begin{aligned}
f^{n-1}(t)=&A^{n-1}(t-R)^{n-1}\big(1+A_2(t-R)^2
\\
&+A_4(t-R)^4+\Odr((t-R)^5)\big),\quad t\to R^{-},
\end{aligned}
\end{align*}
We employ the identity
\begin{equation*}
V(r)=V(R)-w_{n-1}\int\limits_{r}^{R} f^{n-1}(t)\di t,
\end{equation*}
to obtain
\begin{equation}
 \begin{aligned}
 V(t)=&V(R)- \frac{\om_{n-1} A^{n-1}}{n}(t-R)^n
\\ &- \frac{\om_{n-1}A^{n-1}A_2}{n+2}(t-R)^{n+2}+\Odr\big((t-R)^{n+4}\big),
 \end{aligned}\label{6.4}
\end{equation}
and
\begin{equation}
\begin{aligned}
\frac{V^2(t)}{V'(t)}=& B_{n,1}^{(1)}(t-R)^{-n+1}+B_{n,2}^{(1)}(t-R)^{-n+3}+B_{n,3}^{(1)}
(t-R)^{-n+5}
\\
&+B_{n,4}^{(1)}(t-R)+B_{n,5}^{(1)}(t-R)^3+B_{n,6}^{(1)}(t-R)^{(n+1)}
\\
&+\Odr\big((R-t)^{-n+6}+(R-t)^5+(R-t)^{n+3}
\big),\quad t\to R^{-},
\end{aligned}\label{6.5}
\end{equation}
where
\begin{align*}
&B_{n,4}^{(1)}:=-\frac{2V(R)}{A^{n-1}n},\quad B_{n,5}^{(1)}:=\frac{2A_2V(R)}{A^{n-1}}\left(\frac{1}{n}-\frac{A^{n-1}}{n+2}\right), \quad
B_{n,6}^{(1)}:=\frac{\om_{n-1}}{A^{n-1}n^2}.
\end{align*}
We recall that other constants $B_{n,j}^{(1)}$ and $B_j^{(2)}$ were defined in the formulation of Theorem~\ref{th5.1}.

\begin{rem}\label{rm5.1}
Provided $f$ is smooth enough and all its even derivatives vanish at $R$, we can write the next terms in expansion (\ref{6.5}).
They will be of order $\Odr((t-R)^{-n+2k+1})$, $k\geqslant 0$, and for even dimensions we obtain a term of order $\Odr((t-R)^{-1})$.
After integration, this will produce the logarithmic term in the expansion for $G(r)$ as $r\to R^{-}$, giving rise to such
terms appearing in the expansion for $\lambda(r)$.
\end{rem}

Suppose $n\geqslant 6$. Then by (\ref{6.12})
\begin{equation}\label{5.44}
\l(r)=\mu_n(r)\l_1(\r)+\Odr\big((R-r)^{2n-4}\big),\quad r\to R^{-}.
\end{equation}
Let us identify the asymptotic behavior of $\l_1(r)$ as $r\to R^{-}$. We first employ (\ref{6.45}) and integrate (\ref{6.5}) to do it for $G(r)$,
\begin{align*}
G(r)=& B_{n,1}^{(1)}\mu_n^{-1}(r) +B_{n,2}^{(1)}\mu_{n-2}^{-1}(r)+B_{n,3}^{(1)}\mu_{n-4}^{-1}(r) + \Odr\big((R-r)^{-n+7}+1\big),\quad r\to R^{-}.
\end{align*}
Thus, by (\ref{6.40}), (\ref{6.4}) we get
\begin{align*}
\mu_n(r)\l_1(r)=& \frac{V(R)}{B_{n,1}^{(1)}}\mu_n(r) \Bigg(1 - \frac{B_{n,2}^{(1)}}{B_{n,1}^{(1)}}\frac{\mu_n(r)}{\mu_{n-2}(r)}
- \frac{B_{n,3}^{(1)}}{B_{n,1}^{(1)}}\frac{\mu_n(r)}{\mu_{n-4}(r)}
\\
&\Bigg( \frac{B_{n,2}^{(1)}}{B_{n,1}^{(1)}}\frac{\mu_n(r)}{\mu_{n-2}(r)}
\Bigg)^2
\Bigg)+\Odr\big((R-r)^5+(R-r)^{n-2}
\big),\quad r\to R^{-}.
\end{align*}
Together with (\ref{5.44}) it implies (\ref{5.45}) and (\ref{5.46a}).

In order to calculate similar three-terms asymptotics for $\l(r)$ for low dimensions $n=2,3,4,5$ we cannot neglect higher terms of
asymptotics (\ref{6.12}) as in (\ref{5.44}). The reason is that higher terms also contribute to the desired asymptotics. In what follows we
consider separately each dimension.

Consider first the case $n=5$. We take first two terms in (\ref{6.12}),
\begin{equation}\label{5.47a}
\l(r)=\mu_5(r)\l_1(r)+\mu_5^2(r)\l_2(r)+\Odr\big((R-r)^9\big),\quad r\to R^{-}.
\mathcal{}\end{equation}
Integrating (\ref{6.5}) and proceeding as above, as
$r\to R^{-}$ we get
\begin{equation}\label{5.47}
G(r)=-\frac{B_{5,1}^{(1)}}{3} (r-R)^{-3}-B_{5,2}^{(1)} (r-R)^{-1} +B_0^{(2)}+\Odr\big((R-r)^2\big),
\end{equation}
and hence
\begin{equation}\label{5.46}
\begin{aligned}
\mu_5(r)\l_1(r)=& -\frac{3 V(R)}{B_{5,1}^{(1)}}(r-R)^3+ \frac{9V(R)B_{5,2}^{(1)}}{\big(B_{5,1}^{(1)}\big)^2}(r-R)^5
\\
&- \frac{9 V(R) B_0^{(2)}}{\big(B_{5,1}^{(1)}\big)^2}(r-R)^6+\Odr\big((R-r)^7\big),\quad r\to R^{-},
\end{aligned}
\end{equation}
To find out similar formula for $\mu_5^2 \l_2$ in (\ref{5.47a}), we first convert expression (\ref{6.44}) for $\l_2$ to a more convenient form.
Employing (\ref{6.44}), (\ref{6.46}), (\ref{6.30}), and (\ref{6.7}) and integrating by parts, we get
\begin{equation}\label{5.47b}
\begin{aligned}
\mu_n^2(r)\l_2(r)=&-\frac{\mu_n(r)\l_1(r)}{\l_0(r) G(r)}\int\limits_{0}^{r}V'(t)\psi_1(t,r)\di t=\frac{\mu_n(r)\l_1(r)}{\l_0(r) G(r)} \int\limits_{0}^{r}V(t)\psi'_1(t,r)\di t
\\
&= \frac{\mu_n(r)\l_1^2(r)}{\l_0(r) G(r)} \int\limits_{0}^{r} \di t \frac{V(t)}{V'(t)} \int\limits_{0}^{t} V'(s)\psi_0(s,r)\di s
\\
&=\frac{\mu_n^2(r)\l_1^2(r)}{G(r)} \int\limits_{0}^{r} \di t \frac{V(t)}{V'(t)} \int\limits_{0}^{t} \frac{V(s)-V(t)}{V'(s)}V(s)\di s.
\end{aligned}
\end{equation}
We observe that this formula is valid for all $n\geqslant 2$. Together with (\ref{6.4}), (\ref{6.5}), (\ref{5.46}), (\ref{5.47}) it yields
\begin{equation*}
\mu_5^2\l_2(r)=B_1^{(2)}(r-R)^6+\Odr\big((r-R)^7\big),\quad r\to R^{-}.
\end{equation*}
By (\ref{5.47a}), (\ref{5.46}) it implies (\ref{5.49}).

The case $n=4$ can be treated in the same way as the case $n=5$ and below we provide only the main formulas. The analogue of (\ref{5.47}) reads as
\begin{equation*}
G(r)=-\frac{B^{(1)}_{-3}}{2}(r-R)^{-2}+ B_{4,2}^{(1)} \ln (R-r) + B_2^{(2)}+\Odr\big((R-r)^2\big),\quad r\to R^{-},
\end{equation*}
Formulae for $\mu_4\l_1$, $\mu_4^2\l_2$ follow from above one, (\ref{5.47b}),
\begin{align*}
\mu_4(r)\l_1(r)=&-\frac{2V(R)}{B_{4,1}^{(1)}} (r-R)^2 - \frac{4V(R)}{\big(B_{4,1}^{(1)}\big)^2} \big( B_{4,2}^{(1)} \ln (R-r) + B_2^{(2)}\big)
\\
& + \Odr\big( (R-r)^6\ln^2(R-r)
\big),\quad r\to R^{-},
\end{align*}
and
\begin{equation*}
\mu_4^2(r)\l_2(r)=B_3^{(2)}(r-R)^4+\Odr\big((R-r)^6|\ln(R-r)|\big),\quad r\to R^{-}.
\end{equation*}
Hence, by (\ref{6.12}), we obtain (\ref{5.52a}).

We proceed to the case $n=3$. In contrast to all previous cases, here we have to deal with first three terms in (\ref{6.12}), namely,
\begin{equation}\label{5.50}
\l(r)=\mu_3(r)\l_1(r)+\mu_3^2(r)\l_2(r)+\mu_3^3(r)\l_3(r)+\Odr\big((R-r)^4\big), \quad r\to R^{-}.
\end{equation}
First two terms can be treated as above,
\begin{equation*}
G(r)=-B_{3,1}^{(1)}(r-R)^{-1}+B_4^{(2)}+B_{3,2}^{(1)}(r-R)+\Odr\big((R-r)^2\big),
\quad r\to R^{-},
\end{equation*}
and
\begin{align}
&
\begin{aligned}
\mu_3(r)\l_1(r)=&-\frac{V(R)}{B_{3,1}^{(1)}}(r-R)-\frac{V(R)B_4^{(2)}}{\big(
B_{3,1}^{(1)}\big)^2}(r-R)^2
\\
&+B_5^{(2)}(r-R)^3+\Odr\big((R-r)^3\big),\quad r\to R^{-},
\end{aligned}
\label{5.52}
\\
&\mu_3^2\l_2(r)=B_6^{(2)}(r-R)^2+B_7^{(2)}(r-R)^3+\Odr\big((R-r)^4\big),\quad r\to R^{-},\label{5.54}
\end{align}
To obtain similar formula for the third term in the right hand side of (\ref{5.50}), we again first convert formula $\l_3$ by integration by parts, as it was done in (\ref{5.47b}). By (\ref{6.44}) we have
\begin{equation*}
\mu_n^3(r)\l_3(r)=-\frac{\mu_n^2(r)}{\l_0(r)G(r)} \left( \l_1(r) \int\limits_{0}^{r}V'(t)\psi_2(t,r)\di t+ \l_2(r) \int\limits_{0}^{r} V'(t)\psi_1(t,r)\di t\right).
\end{equation*}
Integrating by parts and employing (\ref{6.46}), (\ref{6.7}), we obtain
\begin{align*}
\int\limits_{0}^{r} V'(t)\psi_2'(t,r)\di t&=-\int\limits_{0}^{r} V(t)\psi'_2(t,r)\di t
\\
&=\int\limits_{0}^{r} \di t \frac{V(t)}{V'(t)} \int\limits_{0}^{t} V'(s)\big(\l_2(r)\psi_0(s,r)+\l_1(r)\psi_1(s,r)\big)\di s
\\
&=-\int\limits_{0}^{r} \di t \frac{V(t)}{V'(t)} \int\limits_{0}^{t} \big(V(s)-V(t)\big) \big(\l_2(r)\psi_0'(s,r)+\l_1(r)\psi_1'(s,r)\big)\di s.
\end{align*}
This identity and (\ref{5.47b}) yield
\begin{equation}\label{5.64}
\mu_n^3(r)\l_3(r)=-\frac{\mu_n^3\l_1(r)^3}{G(r)} \int\limits_{0}^{r} \di t \frac{V(t)}{V'(t)} \int\limits_{0}^{t}\di s \frac{V(s)-V(t)}{V'(s)} \int\limits_{0}^{s} \frac{V(z)-V(s)}{V'(z)}V(z)\di z.
\end{equation}
Hence,
\begin{equation*}
\mu_3^3(r)\l_3(r)=B_{10}^{(2)}(r-R)^3+\Odr\big((R-r)^4\big),\quad r\to R^{-},
\end{equation*}
and by (\ref{5.50}), (\ref{5.52}), (\ref{5.54}) it yields (\ref{5.60a}).

It remains to consider the case $n=2$. As in (\ref{5.50}), we take first three terms in (\ref{6.12}),
\begin{equation}\label{5.60}
\l(r)=\mu_2(r) \l_1(r)+\mu_2^2(r)\l_2(r)+\mu_2^3(r)\l_3(r)+\Odr\big(\ln^{-4}(R-r)\big),\quad r\to R^{-}.
\end{equation}
It follows from (\ref{6.5}), (\ref{6.45}) that
\begin{equation}\label{5.61}
G(r)=B_{2,1}^{(1)} \ln(R-r)+B_{11}^{(2)}+\Odr\big((R-r)^2\big),\quad r\to R^{-},
\end{equation}
Hence, by (\ref{6.40}), (\ref{6.4}),
\begin{equation}\label{5.62}
\begin{aligned}
\mu_2(r)\l_1(r)=&\frac{V(R)}{B_{2,1}^{(1)}}\ln^{-1}(R-r) - \frac{ B_{11}^{(2)}V(R)}{\big(B_{2,1}^{(1)}\big)^2} \ln^{-2}(R-r)
\\
&+ \frac{\big(B_{11}^{(2)}\big)^2 V(R)}{\big(B_{2,1}^{(1)}\big)^3} \ln^{-3}(R-r) + \Odr\big((R-r)^2\ln^{-1}(R-r)\big),\quad r\to R^{-}.
\end{aligned}
\end{equation}
Employing (\ref{5.47b}), in the same way we get
\begin{equation}\label{5.63}
\begin{aligned}
\mu_2^2(r)\l_2(r)=&\frac{B_{12}^{(2)}V^2(R)}{\big(B_{2,1}^{(1)}\big)^3} \ln^{-2}(R-r)
\\
& + \left( \frac{B_{13}^{(2)} V^2(R)}{\big(B_{2,1}^{(1)}\big)^3}- 2 \frac{\big(B_{12}^{(2)}\big)^2V^4(R)} {\big(B_{2,1}^{(1)}\big)^6} - \frac{B_{11}^{(2)} B_{12}^{(2)} V^2(R)}{\big( B_{2,1}^{(1)}\big)^4}
\right)\ln^{-3}(R-r)
\\
&+ \Odr\big(\ln^{-4}(R-r)\big),\quad r\to R^{-},
\end{aligned}
\end{equation}
And (\ref{5.64}), (\ref{5.61}), (\ref{5.62}) yield
\begin{equation*}
\mu_2^3(r)\l_3(r)=B_{14}^{(2)}\ln^{-3}(R-r)+\Odr\big(
\ln^{-4}(R-r)\big),\quad r\to R^{-},
\end{equation*}
The last identity and (\ref{5.60}), (\ref{5.62}), (\ref{5.63}) imply (\ref{5.65}).

\section*{Acknowledgments}

D.B. is partially supported by RFBR (grant no. 13-01-00081-a). 
Part of this work was finished while P.F. was visiting the Laboratoire Jacques-Louis Lions at the University Pierre et Marie Curie, Paris, and he would like to
thank the people there for their hospitality.

\end{document}